%% file: main.tex
\documentclass[11pt,reqno]{amsart}
\usepackage{amsmath,amsthm,mathtools,amssymb,bbm}
\usepackage[shortlabels]{enumitem}
\usepackage[margin=1.4in]{geometry}
\usepackage{calrsfs}

\makeatletter
\def\blfootnote{\xdef\@thefnmark{}\@footnotetext}
\makeatother

\usepackage{tikz}
\usetikzlibrary{arrows.meta, calc, decorations.markings}
\usepackage{graphicx} 
\newcommand{\nocontentsline}[3]{}
\newcommand{\tocless}[2]{
  \bgroup
  \let\addcontentsline=\nocontentsline
  #1{#2}
  \egroup
}

\newcommand\eq[1]{\begin{align}{#1}\end{align}}
\newcommand\eqn[1]{\begin{align*}{#1}\end{align*}}
\newcommand\grad[0]{\nabla}
\let\d\relax
\newcommand\d[0]{\partial}
\let\div\relax
\DeclareMathOperator\div{div}

\DeclareMathOperator\supp{supp}

\theoremstyle{plain}
\newtheorem{thm}{Theorem}[section]
\newtheorem{lem}[thm]{Lemma}

\newtheorem{proposition}[thm]{Proposition}

\theoremstyle{definition}
\newtheorem{define}[thm]{Definition}

\theoremstyle{remark}
\newtheorem{remark}[thm]{Remark}

\title[Finite-time blow-up in an elementary model of the 3D NSE]{Finite-time blow-up in an elementary model of the 3D Navier--Stokes equations}
\author{Stan Palasek}
\address{Institute for Advanced Study, 1 Einstein Dr., Princeton, NJ 08540}
\email{palasek@ias.edu}

\begin{document}

\begin{abstract}
We demonstrate finite-time blow-up in a simple, realistic shell model of the 3D Navier--Stokes equations, equipped with ``smooth'' (i.e., rapidly decaying in frequency) initial data and forcing. Previously studied models either exhibit a turbulent cascade that regularizes the three-dimensional viscous dynamics, or rely on highly artificial interactions not transparently realized in the true Euler nonlinearity. We also treat the inviscid, unforced case and obtain singularity formation just above the energy level. We conclude with a discussion of the prospects for embedding the behavior of the dyadic model into the full Euler and Navier--Stokes equations.
\end{abstract}

\maketitle

\section{Introduction}

We are primarily concerned with the problem of global regularity for the Navier--Stokes equations in dimension three, given by
\begin{equation}\begin{aligned}\label{nse}
	\partial_tu-\nu\Delta u+u\cdot\grad u+\grad p=f\\ \div u=0
\end{aligned}\end{equation}
where $\nu>0$ is the viscosity, $u:[0,T]\times\Omega\to\mathbb R^3$ and $p:[0,T]\times\Omega\to\mathbb R$ are unknown velocity and pressure fields, and $f:[0,T]\times\Omega\to\mathbb R^3$ is a prescribed divergence-free external force. Let us take the domain $\Omega=\mathbb T^3$ for the sake of discussion. When equipped with divergence-free initial data $u^0\in C^\infty(\Omega;\mathbb R^3)$ and smooth forcing, a unique smooth solution exists locally in time, as shown originally by Leray~\cite{leray}. The global regularity problem asks whether the maximum lifespan $T_*\in(0,\infty]$ of the smooth solution is ever finite. This question appears quite far from being answered; we refer the reader to~\cite{MR3469428} for recent progress. An analogous question exists for the 3D Euler equations, corresponding to \eqref{nse} with $\nu=0$.

\subsection{Dyadic models and blow-up}\label{dyadic_models_sec}

The problem of finite-time blow-up in the Navier--Stokes equations is notoriously challenging. A plausible strategy toward gaining some insight, if not eventually resolving this question, is to consider simplified models of the equations that capture the same phenomena. A rich class of models are the shell models, also known as dyadic models; see~\cite{MR4607726} for an extensive bibliography. These model the flow of energy between modes via an infinite sequence of coupled ODEs.

In~\cite{MR3486169}, Tao put forth an ambitious strategy to potentially demonstrate a Navier--Stokes blow-up. Very roughly speaking, the program consists of two steps:

\begin{enumerate}[1., leftmargin=1cm]
\item Formulate a shell model of the 3D Navier--Stokes equations (or, if suitable\footnote{Whether blow-up of an \emph{inviscid} shell model suffices for this program depends on the severity of the singularity. For instance, Tao's inviscid blow-up would be adequate, but that of Katz--Pavlovi\'c would not.}, the Euler equations) that exhibits finite-time blow-up.
\item ``Embed'' the dyadic model from Step~1 into the full Navier--Stokes equations. (See the discussion in \S\ref{prospects_section}.)
\end{enumerate}

The remarkable achievement\footnote{Tao in~\cite{MR3486169} addresses Step~2 as well, but only by modifying the Navier--Stokes equations in such a way that the dyadic model from Step~1 is directly manifested in the nonlinearity. The subsequent effort to find a suitable embedding in the full Euler nonlinearity~\cite{MR3809006,MR4196152} faced major obstacles.} of~\cite{MR3486169} was a successful implementation of Step~1, using an ingenious delay mechanism to overcome the regularizing effect of the Kolmogorov-type turbulent cascade. Unfortunately, executing Step~2 with this mechanism appears to be hopelessly far out of reach. Some of the main challenges will be discussed in \S\ref{prospects_section}.

It is noteworthy that a strategy analogous to---but far simpler than---the one proposed by Tao for the blow-up problem \emph{has} recently been implemented toward constructing examples of new phenomena for the Navier--Stokes equations. For the problem of non-uniqueness in certain sharp spaces, Step~1 was carried out by the author in~\cite{MR4992011}, and Step~2 was executed (in part) by the author with M.\ Coiculescu~\cite{MR5008166}. The same strategy was subsequently applied in~\cite{palasek2025arbitrary,cheskidov2025instantaneous}, although the two-step procedure was not articulated explicitly in those cases.

Those results motivate revisiting Tao's strategy toward finite-time blow-up. Let us briefly survey the well-known candidates for Step 1, and their respective limitations:

\begin{description}[leftmargin=.75cm, labelsep=0.3cm]
	\item[Katz--Pavlovi\'c (KP) model~\cite{MR2095627,MR2038114,MR2231615}]For $\lambda>1$, $\alpha\geq1$, and $\nu\geq0$, consider the ODE system \eqn{X_k'=-\nu \lambda^{2k}X_k+\lambda^{\alpha(k-1)}X_{k-1}^2-\lambda^{\alpha k}X_kX_{k+1}.}For the inviscid ($\nu=0$) model with forcing at low modes, there is a turbulent cascade toward a global attractor corresponding to the Kolmogorov-$5/3$ or Onsager regularity. Near the global attractor, the nonlinearity is weak compared to the dissipation when $\alpha<3$. This has a regularizing effect and the model corresponding to Navier--Stokes in dimension three is globally well-posed~\cite{MR2415066,MR2844828}.
	\item[Standard Obukhov model~\cite{obukhov}]With the parameters as above, consider the ODE system \eqn{X_k'=-\nu \lambda^{2k}X_k+\lambda^{\alpha(k-1)}X_{k-1}X_k-\lambda^{\alpha k}X_{k+1}^2.} The inviscid model is globally regular due to the dominance of the ``high-high-low'' interaction when a cascade is underway~\cite{MR2180809}. Very recently, global regularity was proved for the viscous model corresponding to dimension three due to the effect of the high frequencies dissipating away before a cascade has time to arrive~\cite{looi}.
	\item[Tao's model~\cite{MR3486169}] Tao engineered a complicated model with four modes per frequency shell, connected in a delicate way by KP-type interactions, Obukhov-type interactions, as well as a third type (``rotor gates''). This is the only known model that exhibits finite-time blow-up with viscosity in dimension three. A delicate ``clock'' mechanism creates a well-controlled, discrete series of energy transfers which averts both the turbulent cascade seen in the KP model, and the out-of-control growth that occurs in the Obukhov model. The drawbacks of this model include its complexity and the large discrepancy from organic Navier--Stokes interactions, and therefore the likely difficulty in implementing Step 2; see \S\ref{prospects_section} for further discussion.
\end{description}

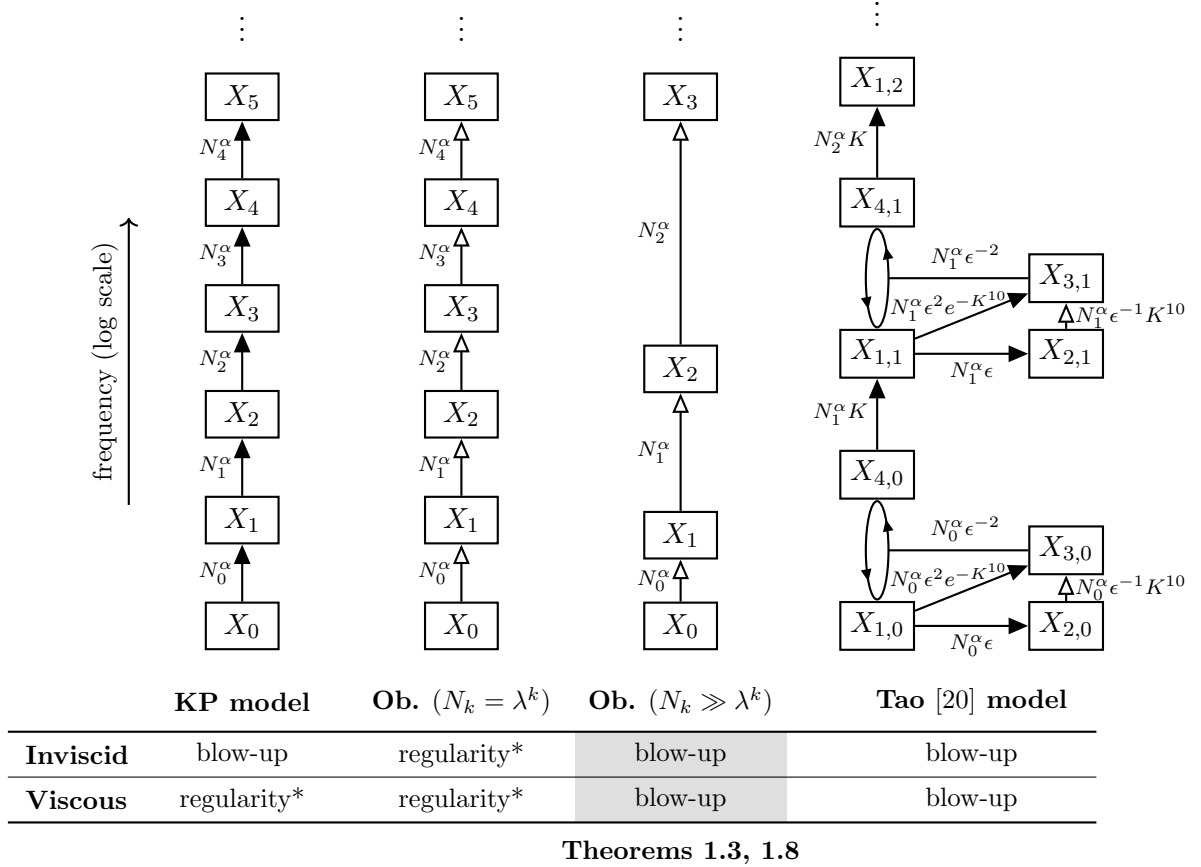
\begin{figure}[htbp]
    \centering
    \input{fig.tex}
    \caption{Comparison of several candidate models of blow-up: the Katz--Pavlovi\'c (KP) model~\cite{MR2095627,MR2038114}, the Obukhov model with exponential or super-exponential frequency separation, and Tao's model from \cite{MR3486169}. We illustrate each model using the quadratic gate notation from \cite[\S5]{MR3486169}. Filled arrows represent KP-type interactions or ``pump gates''; open arrows represent Obukhov-type interactions or ``amplifier gates''; and loops represent ``rotor gates'' whose oscillation rate is proportional to the mode to which the loop is connected.}
    \label{fig:models}
\end{figure}

    \blfootnote{*To our knowledge, regularity is known in these cases only when $f=0$. It is reasonable to conjecture that the regularity mechanisms extend for ``smooth'' forcing. We also emphasize that these regularity claims are specifically for choices of the model parameters corresponding to dimension three.}

We illustrate each of these three standard shell models, as well as the variant considered here, in Figure~\ref{fig:models}. The central purpose of this article is to put forward this new shell model that fulfills Step~1, from which we hope Step~2 might be more tractable. In other words, we seek a shell model that can point to a plausible blow-up scenario for the 3D Navier--Stokes equations for which an embedding in the spirit of \cite{MR5008166} might be practical. 

\subsection{Statement of the model and main theorems}

In this work we consider a certain natural variant of the classical Obukhov model. We defer a detailed discussion of the model's interpretation to \S\ref{motivation_section}.

For $(N_k)_{k\geq0}$ an increasing sequence of frequency scales growing at least exponentially, we say that $X:[0,T]\to\mathbb R^{\mathbb N}$ solves the Obukhov model with external force $f:[0,T]\to\mathbb R^{\mathbb N}$ if
\eq{\label{l2_obukhov}
	X_k'&=-\nu N_k^2X_k+N_{k-1}^\alpha X_{k-1}X_k-N_k^\alpha X_{k+1}^2+f_k
}
where $X_{-1}\equiv0$, with $\alpha\geq1$ a physical parameter related to the intermittency which will be discussed further in \S\ref{motivation_section}. In \eqref{l2_obukhov}, $X_k$ models the $L^2$-norm of the $k$th mode of a Navier--Stokes vector field, localized in frequency around scale $N_k$.

When the frequency scales form a geometric sequence $N_k=\lambda^k$ for some $\lambda>1$, the system \eqref{l2_obukhov} reduces to the standard Obukhov model introduced in~\cite{obukhov} and studied subsequently in \cite{MR2180809}, \cite{MR4992011}, etc. We will find that by generalizing to allow more widely separated sequences of scales, the behavior can change significantly.

\begin{define}[Besov-type spaces]
    Define the Banach space $\mathcal{C}^s=\{X:\mathbb N\to\mathbb R:\|X\|_s<\infty\}$ where
    \eqn{
    \|X\|_{s}\coloneqq\sup_{k\geq0}N_k^s|X_k|.
    }
    Define also the Fr\'echet space $\mathcal{C}^\infty\coloneqq\cap_{s>0}\mathcal{C}^s$.
\end{define}

Because the model assumes a particular fractal structure for the spatial support of the fluid (see \S\ref{motivation_section}), measuring $X_k$ in $\mathcal C^s$ is analogous to measuring the velocity in the Besov scale $B_{p,\infty}^{s+(\frac2p-1)(\alpha-1)}(\mathbb T^3)$ for any $p\in[1,\infty]$. Furthermore, $\mathcal C^\infty$ corresponds to the space of smooth functions on $\mathbb T^3$.

The space $\mathcal C^s$ depends on the choice of $N_k$, for which we now fix
\eq{\label{nk_choice}
N_k=N_0^{b^k}
}
for some $N_0>1$ and $b>1$ to be specified in the course of the proofs in \S\ref{proof_section}.  

Having established the functional setting, we now state the local theory for the viscous and inviscid Obukhov model with smooth initial data and forcing. If $X^0\in \mathcal{C}^\infty$ and $f\in C_t^\infty([0,\infty);\mathcal C^\infty)$, then there exist $T>0$ and a unique solution $X\in \cap_{s>0}C_t([0,T);\mathcal{C}^s)$. Moreover, there exists a maximal time of existence $T_*\in (0,\infty]$. These claims follow from completely standard fixed point arguments which we do not expound upon.

\begin{define}
    We say finite-time blow-up occurs in a solution $X(t)$ with initial data $X^0\in \mathcal{C}^\infty$ and force $f\in C_t^\infty([0,\infty);\mathcal C^\infty)$ if $T_*<\infty$.
\end{define}

Our first main result is that, for $N_k$ defined by \eqref{nk_choice}, the viscous Obukhov model exhibits finite-time blow-up from $\mathcal C^\infty$ initial data under a $C_t^\infty\mathcal C^\infty$ force.

\begin{thm}[Viscous blow-up]\label{viscous_ns_blowup_theorem}
    Let $\nu>0$, $\alpha>2$, and $N_k$ as in \eqref{nk_choice}. Then there exist positive initial data $X^0\in \mathcal{C}^\infty$ and force $f\in C_t^\infty([0,\infty);\mathcal{C}^\infty)$ that give rise to a finite-time blow-up of \eqref{l2_obukhov}. In particular, for any $s>0$, the data can be chosen so that $X$ becomes unbounded in $\mathcal{C}^{s}$.
\end{thm}

Some remarks on the strength and context of the viscous theorem:

\begin{remark}
The inclusion of an external force is necessary for blow-up in this model, as shown in~\cite{looi}. In Theorem~\ref{viscous_ns_blowup_theorem}, the external force is smooth and has no role at the blow-up time. In fact,
\eqn{
\| f(t) \|_{\mathcal{C}^s} \to 0 \quad \text{as } t \to T_*, \quad \forall s > 0.}
\end{remark}

\begin{remark}
The permissible range $\alpha>2$ is completely sharp---when $\alpha\leq2$, the energy is critical or better, leading to global well-posedness from $\mathcal C^\infty$ data. Furthermore, note that $\alpha\in[1,5/2]$ corresponds to the intermittency range relevant in three spatial dimensions; therefore, the window $\alpha\in(2,5/2]$ offers candidate blow-up scenarios for 3D Navier--Stokes.
\end{remark}

\begin{remark}
While the exact double exponential form in \eqref{nk_choice} is not crucial, it is probably necessary for blow-up that $(N_k)_{k\geq1}$ grows faster than exponentially; see~\cite{MR2180809}. See, for instance,~\eqref{exp_small} and \eqref{ratios} below where we make essential use of this fact.
\end{remark}

\begin{remark}
    Unboundedness in $\mathcal C^s$ for $s>0$ is sharp in light of the fact that $\mathcal C^0$ is controlled by the energy (i.e., $\ell^2(\mathbb N)$); see also the heuristic discussion in \S\ref{inviscid_heuristics}. Analogizing to the PDE setting, the solution would be unbounded in $B^{\epsilon}_{2,\infty}$ and $B^0_{2+\epsilon,\infty}$, just above $L^2$.

    One can see in particular that the blow-up is Type II in the sense that it exits the critical space $\mathcal C^{\alpha-2}$. Indeed, in the notation of \S\ref{proof_section}, the solution approaches the blow-up state $X_k=N_k^{\beta-\alpha}$ for some $\beta>2$.
\end{remark}

Next we state a result for the inviscid case.

\begin{thm}[Inviscid blow-up]\label{inviscid_euler_blowup_theorem}
    Let $\nu=0$, $\alpha\geq1$, and $N_k$ as in \eqref{nk_choice}. There exists positive initial data $X^0\in \mathcal{C}^\infty$ that gives rise to a finite-time blow-up of \eqref{l2_obukhov} with $f=0$. In particular, for any $s>0$, the data can be chosen so that $X$ becomes unbounded in $\mathcal{C}^{s}$.
\end{thm}

\begin{remark}
    Various blow-up scenarios have been demonstrated before in inviscid shell models; see for instance~\cite{MR2095627,MR2038114,MR2231615,MR2180809,MR2337019,MR2600714} for the Katz--Pavlovi\'c and similar models, as well as Tao's model~\cite{MR3486169}. Our Theorem~\ref{inviscid_euler_blowup_theorem} is the first\footnote{We point out that in~\cite{MR3339169}, Jeong and Li prove blow-up for a mixed Obukhov--KP system with a small parameter in front of the Obukhov part; thus the KP behavior is dominant in their scenario.} for an Obukhov-type system.
\end{remark}

\begin{remark}
    The blow-up is unstable in every $\mathcal C^s$. Indeed, if one truncates the initial data above some $k_*$, the solution stays supported below $k_*$ for all $t>0$. Since $X^0\in\mathcal C^\infty$, such a truncation can be made arbitrarily small in $\mathcal C^s$ with a large enough choice of $k_*$.
\end{remark}

\subsection{Blow-up mechanism}

In this section, we summarize the blow-up mechanism and heuristically justify the parameters and constraints that appear in the main theorems. See Figure~\ref{fig} for an illustration of the blow-up.

\subsubsection{Inviscid case}\label{inviscid_heuristics} The blow-up arises from the dominance of the low-high-high interaction, with the opposing high-high-low term having negligible influence. Consider the simplified (non-energy-conserving) model
\eqn{
X_k'&= N_{k-1}^\alpha X_{k-1}X_k, \qquad k\geq0
}
with the usual convention that $X_{-1}\equiv0$. Suppose we want the solution to arrive in finite time at a singular profile $X_k\sim A_k$ where\footnote{We warn that in the proofs (\S\ref{proof_section}), we work with certain rescaled variables; thus the $A_k$ appearing in this section corresponds to $N_k^{-\alpha} A_k$ in that notation. In other words, $\gamma=\beta-\alpha$, with $\beta$ as in \S\ref{proof_section}.} $A_k=N_k^\gamma$ for some $\gamma<0$. Observe that $X_0(t)$ is completely stationary, so we take initially $X_0^0=A_0$ with the rest of the data positive and decaying rapidly in $N_k$. Clearly then $X_1(t)=X_1^0\exp(N_0^\alpha A_0t)$; thus it arrives at its destination at $t\approx T_1$ where we define
\eqn{
T_k\coloneqq (N_{k-1}^\alpha A_{k-1})^{-1}\log\frac{A_k}{X_k^0}.
}
Note also that $X_1(t)$ is nearly stationary on the time interval $I_1$, where we set
\eqn{
I_k=\{t:|t-T_k|\ll (N_{k-1}^\alpha A_{k-1})^{-1}\}.
}
Thus, for $t\in I_1$, we may assume $X_1(t)\approx A_1$ which leads to, roughly speaking,
\eqn{
X_2'\approx \mathbbm1_{I_1}(t)N_1^\alpha A_1X_2,
}
so $X_2(t)$ grows exponentially in this time window and would arrive at $A_2$ after growing for approximately time $T_2$. The claim is that this type of growth can be continued for all $k\geq1$. A clear constraint is that the growth time for $X_k$ must stay within the stability interval $I_{k-1}$ for $X_{k-1}$; thus
\eqn{
T_k\ll (N_{k-2}^\alpha A_{k-2})^{-1}.
}
This effectively sets a constraint on $X_k^0$:
\eqn{
X_k^0\geq A_k\exp\Big(-c\frac{N_{k-1}^\alpha A_{k-1}}{N_{k-2}^\alpha A_{k-2}}\Big).
}
Here one already sees why the super-exponential frequency scale is indispensable: if $N_k$ and $A_k$ depended \emph{exponentially} on $k$, then the right-hand side would be at best a small multiple of $A_k$; thus, the best one could hope for would be a norm inflation-type result, not a genuine loss of regularity. With instead double exponential dependence, we have $\frac{N_{k-1}^\alpha A_{k-1}}{N_{k-2}^\alpha A_{k-2}}=N_k^{c}$ for some $c>0$, permitting $\mathcal C^\infty$ initial data.

Another constraint that must be verified is that the cumulative effect of the high-high-low term $-N_k^\alpha X_{k+1}^2$ is negligible. Consider the time interval $I_{k+1}$ when $X_{k+1}$ is growing. Continuing the heuristic calculations above, we may assume
\eqn{
X_{k+1}(t)\approx X_{k+1}^0\exp(N_{k}^\alpha A_{k}(t-t_*))
}
for some suitable start time $t_*$. Then the cumulative effect of the negative interaction on $X_k$ is roughly
\eqn{
\int_{I_{k+1}}-N_k^\alpha X_{k+1}^2(t)dt\approx -A_k^{-1} (X_{k+1}^0)^2\exp(2N_k^\alpha A_kT_{k+1})\approx -A_k^{-1}A_{k+1}^2,
}
having used the definition of $T_{k+1}$. For this to have a negligible effect on $X_k$, which reaches size $\sim A_k$, we should have $A_{k+1}\ll A_k$. If $A_k=N_k^\gamma$, then this requires $\gamma<0$, and we have for the blow-up profile
\eqn{
\|X(T_*)\|_{\mathcal C^s}\sim \sup_{k\geq0} N_k^{\gamma+s},
}
so $X$ becomes unbounded in $\mathcal C^s$ if $s>-\gamma$. Since $\gamma<0$ can be chosen freely, this heuristically justifies the necessity of $s>0$ in Theorems~\ref{viscous_ns_blowup_theorem} and \ref{inviscid_euler_blowup_theorem}.

We emphasize that the above heuristic discussion contains several shortcomings and cannot be directly converted to a proof. For instance, it is not accurate to say that $X_k(t)$ begins growing only when $X_{k-1}(t)$ becomes comparable to its target size $A_{k-1}$, even when using $N_0\gg1$ as a large parameter to keep the scale separation large. At a technical level, it is necessary to designate the desired blow-up profile $X_k(0)=A_k$ and evolve the (truncated) system \emph{backward in time}. Using a trapping region argument, one can justify that the backward solution obeys uniform estimates, allowing passage to a limit, and that the limit is regular at some $t=-T$. Similar arguments involving local backward-in-time continuation from the blow-up have appeared before, for instance in~\cite{MR4929621}.

\begin{figure}
    \centering
    \includegraphics[width=.8\linewidth]{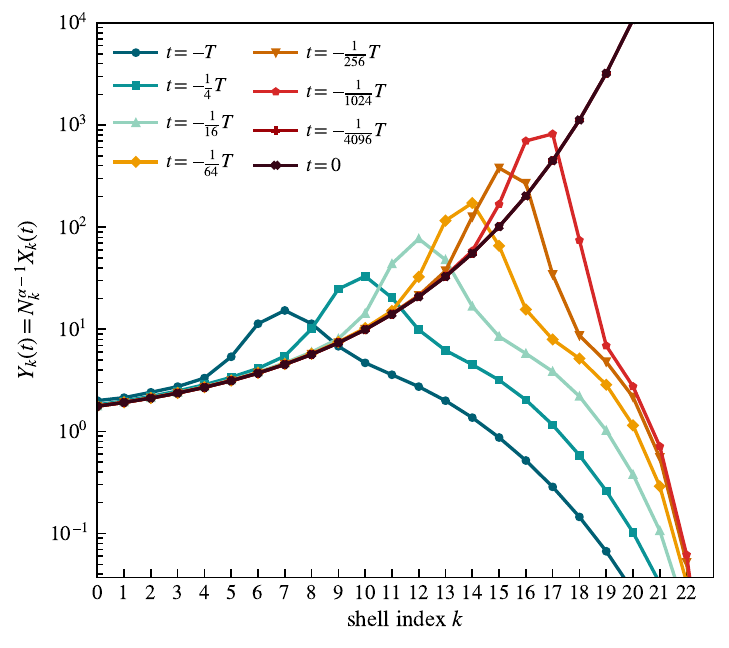}
    \caption{A blow-up of the viscous Obukhov model at $t=0$. We depict snapshots of the solution $Y_k(t)=N_k^{\alpha-1}X_k(t)$, which models $\|P_ku(t)\|_{L^\infty}$, at various times $t\in[-T,0]$, including the blow-up profile $Y_k(0)=N_k^{1.4}$ (black). The system \eqref{l2_obukhov} is simulated with $\nu=1$, $\alpha=\frac52$, and $N_k=1.5^{1.15^k}$.}
    \label{fig}
\end{figure}

\subsubsection{Viscous case}\label{viscous_discussion}

The viscous case is similar, with the new constraint that the damping term should be negligible during the growth interval $I_k$. The competition between the two effects is captured as
\eq{\label{approx}
X_k'\approx (N_{k-1}^\alpha X_{k-1}-\nu N_k^2)X_k
}
with the high-high-low interaction arranged to be negligible as discussed in \S\ref{inviscid_heuristics}. As explained above, $X_{k-1}\big|_{I_k}\approx A_{k-1}$, so growth requires $A_{k-1}\gg N_k^2 N_{k-1}^{-\alpha}$. Recall from above that we must have $A_{k-1}\ll1$, leading to the requirement $\alpha>2$. In other words, the blow-up occurs in the energy supercritical range.

One encounters another constraint in the viscous setting which is not connected to energy criticality, but is nonetheless unavoidable as recently shown by Looi~\cite{looi}. Consider \eqref{approx} with positive initial data. Then growth or decay of $X_k(t)$ is governed by the sign of $-\nu N_k^2t+N_{k-1}^\alpha\int_0^tX_{k-1}$. By Cauchy--Schwarz and the energy equality, the second term is $O(N_{k-1}^{\alpha-1})$, so growth is obstructed when $\alpha<3$. In essence, each mode $X_{k-1}$ (for $k$ large) is only \emph{required} to be active over the short time span $T_{k}$ during which it causes growth of $X_{k}$ (see above), but, if arising from the initial data, is forced to be active for the full solution lifespan. To avoid this issue, we use a small force to deactivate the dissipation of each mode for times prior to its activation. Because the initial data is so regular, the high modes begin very small; thus this deactivation can be achieved with a very regular force. Furthermore, the forcing at each mode can be arranged to be supported away from the blow-up time. At each frequency level, the force in a sense sets the stage for the blow-up, but has no role in the growth itself.

\textbf{Plan of the paper:} In \S\ref{motivation_section}, we explain in detail how the shell model \eqref{l2_obukhov} corresponds to the full Navier--Stokes equations. In \S\ref{proof_section}, we prove the blow-up for the viscous and inviscid versions of the system. Finally, in \S\ref{prospects_section}, we discuss the advantages of pursuing a Navier--Stokes blow-up via an embedding of \eqref{l2_obukhov} and compare its utility to the other well-known shell models.

\tocless{\section*}{Acknowledgments}

This work was supported by the National Science Foundation under Grant No. DMS-2424441. This work benefited from discussions with Sam Looi. GPT-5.4 assisted with generating figures.

\medskip

\section{Motivating the shell model}\label{motivation_section}

In this section we motivate the use of the Obukhov model \eqref{l2_obukhov} and connect its terms with particular interactions within the Navier--Stokes equations. An extended discussion is warranted because the model has been studied very little in general, and, to our knowledge, has never been considered in the general case with frequency scales $N_k$ growing faster than exponentially. This omission presumably arose because shell models have conventionally been used to model generic turbulent flows, for which $N_k=\lambda^k$ is the only reasonable choice. However, recent constructions of counterexamples in mathematical fluids have shown that a more lacunary frequency support can be advantageous.

Recall the full incompressible Navier--Stokes and Euler equations, given in \eqref{nse}. We are concerned with a scenario in which the velocity can be approximately decomposed as
\eqn{
u=\sum_{k\geq0}u_k,\qquad\supp\hat u_k\subset\{\xi\in\mathbb Z^3:|\xi|\sim N_k\}
}
for some rapidly increasing sequence of frequency scales $(N_k)_{k\geq0}$. To reduce to a shell model, one must choose which nonlinear interactions between the $u_k$ to retain and which to discard. The Obukhov model arises after assuming $u_k$ is acted upon only by the interactions\footnote{This pair of interactions is permissible because it conserves energy: the energy transferred up by the first term is balanced by energy transferred down by the second term. There is one other pair of nearest-neighbor interactions with this property, which corresponds to the Katz--Pavlovi\'c model.} $u_{k}\cdot\grad u_{k-1}$ and $u_{k+1}\cdot\grad u_{k+1}$. To justify this choice, it is important to point out that these interactions affect $u_k$ even when $(N_k)_{k\geq0}$ grows super-exponentially. Indeed, if $\hat u_{k+1}$ is supported in $\{\xi\in\mathbb Z^3:|\xi|\sim N_{k+1}\}$, then $\widehat{u_{k+1}\otimes u_{k+1}}$ can occupy all of $B(0,O(N_{k+1}))$.

We model the interactions by scalar functions $X_k(t)$ and $Y_k(t)$ which correspond to the $L^2$ and $L^\infty$ norms, respectively, of the velocity, projected to frequency scale $N_k$:
\eqn{
X_k\sim\|P_{k}u\|_{L^2},\qquad Y_k\sim\|P_{k}u\|_{L^\infty}
}
for $k\geq0$, with $N_k>0$ a rapidly increasing sequence. Here and elsewhere, $P_k$ is used heuristically to refer to a Littlewood--Paley-type projection to frequencies of magnitude near $N_k$. We sometimes abbreviate $P_{k}u$ by simply $u_k$.

A fundamental assumption of the model is that the PDE solution mainly concentrates on a fractal set where the component localized near the $k$th frequency scale occupies a volume $\sim N_k^{-2(\alpha-1)}$. Thus, $X_k$ and $Y_k$ are related by
\eq{\label{X_Y_relation}
X_k=N_k^{-(\alpha-1)}Y_k.
}
The intermittency parameter $\alpha$ is constrained from below by the volume of the torus and from above by the uncertainty principle, leading to the range $\alpha\in[1,\frac52]$ corresponding to the 3D Navier--Stokes equations.

In $L^\infty$, it is straightforward to see that the effect of $u_{k}\cdot\grad u_{k-1}$ at frequency level $k$ is, roughly, on the order $N_{k-1}Y_{k-1}Y_k$. The effect of the term $u_{k+1}\cdot\grad u_{k+1}$ is more subtle, particularly in the case of long range interactions ($N_{k+1}\gg N_k$). To project the interaction to the $k$th frequency shell, one effectively averages on length scale $\sim N_k^{-1}$. This averaging entails multiplying by the volume of the support of $u_{k+1}$ and dividing by the volume of the support of $u_k$, leading to the extra factor $N_{k+1}^{-2(\alpha-1)}/N_{k}^{-2(\alpha-1)}$. Moreover, because the interaction can be written in divergence form as $\div u_{k+1}\otimes u_{k+1}$, one sees that the derivative falls on the lower frequency shell, producing a factor of $N_k$. In total, we arrive at the $L^\infty$-based Obukhov model
\eq{\label{linfty_obukhov}
	Y_k'=-\nu N_k^2Y_k+N_{k-1}Y_{k-1}Y_k-\Big(\frac{N_k}{N_{k+1}}\Big)^{2(\alpha-1)}N_kY_{k+1}^2+h_k(t),
}
having introduced an external force $h_k(t)$. Utilizing the relation \eqref{X_Y_relation} between $X_k$ and $Y_k$, one easily sees that \eqref{linfty_obukhov} is equivalent to the $L^2$-based Obukhov model \eqref{l2_obukhov} with an appropriately scaled force.

We mention that the equivalent systems \eqref{l2_obukhov} and \eqref{linfty_obukhov} share many key properties with the true Navier--Stokes system. For instance, defining the energy
\eqn{
e(t)&=\frac12\sum_k X_k^2(t)=\frac12\sum_k N_k^{-2(\alpha-1)}Y_k^2(t),
}
we have the identity
\eqn{
e(t)+\int_0^t\sum_k \nu N_k^2X_k^2(s)ds=e(0)+\int_0^t\sum_kX_kf_k
}
for solutions of \eqref{l2_obukhov} above a certain Onsager-like regularity and sufficiently regular forcing. When $N_k=\lambda^k$ for some $\lambda>1$ and $k$ runs over $\mathbb Z$, there is a discrete group of scaling symmetries
\eqn{
Y_k(t)\mapsto \lambda^n Y_{k-n}(\lambda^{2n}t),\quad n\in\mathbb Z
}
corresponding to a discrete subgroup of the usual scaling symmetries for the Navier--Stokes equations on $\mathbb R^3$. There is also a theory of global Leray solutions from $\ell^2$ data, where uniqueness is known to break down when $N_k=\lambda^k$ for some $\lambda>1$ large~\cite{MR4992011}. These parallels motivate the use of the model to gain insights about the full Navier--Stokes system.

\section{Proof of blow-up theorems}\label{proof_section}

Toward proving Theorems~\ref{viscous_ns_blowup_theorem} and~\ref{inviscid_euler_blowup_theorem}, it will be profitable to consider the rescaled unknown $x$ defined by
\eqn{
X_k=N_k^{-\alpha}x_k.
}
Dimensionally, $x$ corresponds to (for instance) $\|P_k\omega\|_\infty$ and \eqref{l2_obukhov} transforms into the system
\eq{\label{x_system_nondimensionalized}
x_k'=-\nu N_k^2x_k+x_{k-1}x_k-\delta_kx_{k+1}^2+g_k
}
for $k\geq0$, where
\eq{
\delta_k\coloneqq \Big(\frac{N_k}{N_{k+1}}\Big)^{2\alpha}
}
and $g_k\coloneqq N_k^\alpha f_k$. As before, we fix $x_{-1}\equiv0$. The new system \eqref{x_system_nondimensionalized} has the clear advantage that the growth rate of $x_k$ is precisely $x_{k-1}$, while the unfavorable high-high-low term is scaled by the small dimensionless parameter $\delta_k$.

By translation, we construct a blow-up at $t=0$ from data at $t=-T$. By rescaling, we may reduce to the two cases $\nu=0$ and $\nu=1$.

\subsection{The trapping region}\label{trapping_subsec}

Let $s>0$ as in Theorems~\ref{viscous_ns_blowup_theorem} and \ref{inviscid_euler_blowup_theorem}. Let $A_k=N_k^\beta$ be an increasing sequence of amplitudes where $b$ and $\beta$ are chosen to satisfy
\eq{\label{inviscid_A_assumptions}
b>1,\qquad \max\{0,\alpha-s\}<\beta<\alpha
}
if $\nu=0$, and
\eq{\label{viscous_A_assumptions}
b\in\Big(1,\frac\alpha2\Big),\qquad \max\{2b,\alpha-s\}<\beta<\alpha
}
if $\nu>0$. We also fix a small constant $c>0$ and define
\eq{\label{tk_times_def}
T\coloneqq\frac{c}{A_0},
\qquad
t_1\coloneqq t_2\coloneqq -T,
\qquad
t_k\coloneqq -cA_{k-2}^{-1}\quad (k\geq 3).
}
Since $A_k$ is increasing, we have
\eqn{
-T=t_1=t_2<t_3<t_4<\cdots<0.
}
Having defined the relevant parameters, let us record the following elementary facts: first, for any choice of parameters satisfying $c_3,c_4>0$, we may take $N_0>1$ large to arrange that
\eq{\label{exp_small}
\left(\frac{N_{k+i_1}}{N_{k+i_2}}\right)^{c_1}\left(\frac{A_{k+i_3}}{A_{k+i_4}}\right)^{c_2}\exp(-c_3\left(\frac{A_{k+1}}{A_k}\right)^{c_4})\text{ is arbitrarily small},
}
uniformly in $k\geq0$. Second, we can arrange that
\eq{\label{ratios}
\frac{A_k}{A_{k-1}}\leq\frac c{100}\frac{A_{k+1}}{A_k}\quad\forall k\geq1.
}
We point out that these statements make essential use of the fact that $(N_k)_{k\geq0}$ grows faster than exponentially, since otherwise, $A_{k+1}/A_k$ would be bounded.

The proof proceeds by constructing a trapping region for Galerkin-truncated solutions of \eqref{x_system_nondimensionalized} backward in time from $t=0$. The region itself is defined as follows. 

\begin{define}\label{R_region_definition}
Fix $K\in\mathbb N$. Define upper barriers $\zeta_k$ and lower barriers $\eta_k$ on
$[-T,0]$ as follows: for $1\leq k\leq K$, let $\zeta_k$ be the unique function satisfying
\eqn{
\zeta_k(0)=A_k
}
and
\[
\zeta_k'(t)=\begin{cases}
0, & t\in[-T,t_k),\\
\frac1{2}A_{k-1}\zeta_k(t)-\delta_k\zeta_{k+1}(t)^2\,\mathbbm 1_{k\leq K-1},
& t\in[t_k,0].
\end{cases}
\]
Then, let $\zeta_0$ be such that
\eqn{
\zeta_0'=-\nu N_0^2\zeta_0-\delta_0\zeta_1^2,
\qquad
\zeta_0(0)=A_0.
}
Next define
\eqn{
\eta_0(t)\coloneqq A_0,
\qquad
\eta_k(t)\coloneqq A_k\exp\Bigl(-\int_t^0\zeta_{k-1}(s)\,ds\Bigr),
\qquad 1\leq k\leq K.
}
Finally, define the time-dependent rectangle
\eqn{
\mathcal R_K(t)
\coloneqq \Bigl\{(x_0,\dots,x_K)\in\mathbb R^{K+1}:
\eta_k(t)\leq x_k\leq \zeta_k(t)\ \text{for all }0\leq k\leq K
\Bigr\}.
}
\end{define}

We show the following elementary estimates for the upper and lower limits of the region $\mathcal R_K$. These will amount to bounds on any solution confined to the region.

\begin{lem}[Barrier bounds]\label{barrier_bounds_lemma}
Let $t_k$ be as in \eqref{tk_times_def} and $\eta_k,\zeta_k$ as in Definition~\ref{R_region_definition}. Assume also the relevant parameter inequalities, namely \eqref{inviscid_A_assumptions} in the inviscid case and \eqref{viscous_A_assumptions} in the viscous case. Then for every $1\leq k\leq K$ and $t\in[-T,0]$,
\begin{equation}\label{z_bound}
A_k\exp\Bigl(\frac12A_{k-1}\max\{t,t_k\}\Bigr)\leq \zeta_k(t)\leq 2A_k\exp\Bigl(\frac12A_{k-1}\max\{t,t_k\}\Bigr).
\end{equation}
Moreover,
\begin{equation}\label{z0_bound}
A_0\leq \zeta_0(t)\leq 2A_0
\qquad \forall\, t\in[-T,0].
\end{equation}
For the lower limits, we have
\begin{equation}\label{eta_quarter_bound}
\eta_k(t)\geq \frac34A_k
\qquad
\forall\,0\leq k\leq K-1,\ \forall\, t\in[t_{k+1},0],
\end{equation}
and for every $2\leq k\leq K$,
\begin{equation}\label{eta_global_bound}
\eta_k(t)\geq A_k\exp\Bigl(-5\frac{A_{k-1}}{A_{k-2}}\Bigr)
\qquad
\forall\, t\in[-T,0].
\end{equation}
\end{lem}

\begin{proof}
For $k=K$ we have an explicit formula
\eqn{
\zeta_K(t)=A_K\exp\Bigl(\frac12A_{K-1}\max\{t,t_K\}\Bigr).
}
This gives \eqref{z_bound} for $k=K$.

Now fix $1\leq k\leq K-1$. Since both $\zeta_k$ and the bounds claimed in \eqref{z_bound} are constant on $[-T,t_k]$, it suffices to
estimate it on $[t_k,0]$. On this interval, we have the Duhamel formula
\eqn{
\zeta_k(t)=A_ke^{\frac12A_{k-1}t}+\mathcal N_k(t),\quad\mathcal N_k(t)\coloneqq
\delta_k\int_t^0
e^{-\frac12A_{k-1}(s-t)}\zeta_{k+1}(s)^2\,ds.
}
The lower bound in \eqref{z_bound} is immediate. For the upper bound, assume inductively that \eqref{z_bound} holds for $k+1$, recalling that the $k=K$ estimates have already been established. We distinguish two cases.

\smallskip
\noindent
\textit{Case 1: $t\in[t_k,t_{k+1})$.}
Using the induction hypothesis and splitting at $t_{k+1}$, we obtain
\begin{align*}
\mathcal N_k(t)&\lesssim
\delta_kA_{k+1}^2
\int_t^{t_{k+1}}
e^{-\frac12A_{k-1}(s-t)}e^{A_kt_{k+1}}\,ds+\delta_kA_{k+1}^2
\int_{t_{k+1}}^0
e^{-\frac12A_{k-1}(s-t)}e^{A_ks}\,ds.
\end{align*}
The first term is controlled by
\eqn{
\delta_k A_{k+1}^2|t_k|e^{\frac12A_{k-1}t}e^{A_kt_{k+1}-\frac12A_{k-1}t}\lesssim \left(\frac{\delta_kA_{k+1}^2}{A_kA_{k-2}}e^{-cA_k/A_{k-1}+\frac c2A_{k-1}/A_{k-2}}\right)A_ke^{\frac12A_{k-1}t}
}
where the first factor can be made arbitrarily small with the choice of $N_0$, using \eqref{exp_small} and \eqref{ratios}. We point out that if $k\leq1$, then $t_k=t_{k+1}=-T$ so this case is vacuous; thus we may assume $k\geq2$, under which assumption the quantity $A_{k-2}$ is well-defined. 

For the second term,
\eqn{
\delta_kA_{k+1}^2
\int_{t_{k+1}}^0
e^{-\frac12A_{k-1}(s-t)}e^{A_ks}\,ds
\lesssim
\left(\delta_k\frac{A_{k+1}^2}{A_k^2}\right)A_ke^{\frac12A_{k-1}t}.
}
By definition, we have
\eqn{
\delta_k\frac{A_{k+1}^2}{A_k^2}&=\left(\frac{N_{k+1}}{N_k}\right)^{-2\alpha+2\beta}=N_0^{-2(\alpha-\beta)(b-1)b^k}. 
}
Using the assumption $\beta<\alpha$ from \eqref{inviscid_A_assumptions} and \eqref{viscous_A_assumptions}, we find that this quantity too can be made small with the choice of $N_0$. Combining the two bounds on $\mathcal N_k(t)$,
\eqn{
\zeta_k(t)\leq (1+o_{N_0\to\infty}(1))A_ke^{\frac12A_{k-1}t}
}
and we reach the desired conclusion.

\smallskip
\noindent
\textit{Case 2: $t\in[t_{k+1},0]$.}
Here $\max\{s,t_{k+1}\}=s$ for every $s\in[t,0]$, so by the induction hypothesis,
\begin{align*}
\mathcal N_k(t)&\lesssim
\delta_kA_{k+1}^2
\int_t^0e^{-\frac12A_{k-1}(s-t)}e^{A_ks}\,ds\lesssim
\left(\delta_k\frac{A_{k+1}^2}{A_k^2}\right)A_ke^{\frac12A_{k-1}t}.
\end{align*}
Again, after increasing $N_0$ if necessary,
\eqn{
\zeta_k(t)\leq 2A_ke^{\frac12A_{k-1}t}
\qquad \forall\, t\in[t_{k+1},0].
}
This completes the downward induction and proves \eqref{z_bound}.

For $k=0$, integrating the equation for $\zeta_0$ and using \eqref{z_bound} with
$k=1$, we get
\begin{align*}
\zeta_0(t)&=e^{\nu N_0^2(-t)}A_0+\delta_0\int_t^0e^{\nu N_0^2(s-t)}\zeta_1(s)^2\,ds\\
&\leq e^{\nu N_0^2T}\left(A_0+4\delta_0A_1^2\int_{-T}^0e^{A_0s}\,ds\right)\\
&\leq e^{\nu cN_0^2/A_0}\left(A_0+4\frac{\delta_0A_1^2}{A_0}\right)
\end{align*}
and conclude in the same manner. Note that in the viscous case $\nu=1$, we use \eqref{viscous_A_assumptions} to estimate $cN_0^2/A_0= cN_0^{2-\beta}\leq1/10$, say. 

We now estimate $\eta_k$. By definition,
\eqn{
\eta_k(t)=A_k\exp\Bigl(-\int_t^0\zeta_{k-1}(s)\,ds\Bigr).
}
To prove \eqref{eta_quarter_bound}, fix $1\leq k\leq K-1$ and $t\geq t_{k+1}$. (The $k=0$ case is trivial.)
If $k=1$, then by \eqref{z0_bound},
\eqn{
\int_t^0\zeta_0(s)\,ds\leq 2A_0(0-t)\leq 2A_0T=2c
}
which can be made as small as needed by the choice of $c>0$.

Now let $k\geq 2$. Using \eqref{z_bound},
\eqn{
\int_t^0\zeta_{k-1}(s)\,ds\leq 2A_{k-1}\int_t^0e^{\frac12A_{k-2}s}\,ds=4\frac{A_{k-1}}{A_{k-2}}
\Bigl(1-e^{\frac12A_{k-2}t}\Bigr).
}
Since $t\geq t_{k+1}=-cA_{k-1}^{-1}$,
\eqn{
1-e^{\frac12A_{k-2}t}\leq 1-e^{-\frac c2\frac{A_{k-2}}{A_{k-1}}}\leq \frac c2\frac{A_{k-2}}{A_{k-1}}.
}
Therefore
\eqn{
\int_t^0\zeta_{k-1}(s)\,ds\leq 2c,
}
so
\eqn{
\eta_k(t)\geq A_ke^{-2c},
}
and we conclude \eqref{eta_quarter_bound} once again by the choice of $c>0$.

Finally, for the global-in-time lower bound, let $k\geq 2$ and $t\in[-T,0]$. By positivity, we may split
\eqn{
\int_t^0\zeta_{k-1}(s)\,ds\leq\mathbbm1_{t<t_{k-1}}\int_t^{t_{k-1}}\zeta_{k-1}(s)\,ds+\int_{t_{k-1}}^0\zeta_{k-1}(s)\,ds.
}
If $k=2$, then $t_{k-1}=t_1=-T$, so the first term is absent. For the second term,
\eqn{
\int_{-T}^0\zeta_1(s)\,ds\leq 2A_1\int_{-T}^0e^{A_0s/2}\,ds\leq 4\frac{A_1}{A_0},
}
which proves \eqref{eta_global_bound} for $k=2$.

Assume now that $k\geq 3$. Since $\zeta_{k-1}$ is constant on $[-T,t_{k-1}]$,
\eqn{
\int_t^{t_{k-1}}\zeta_{k-1}(s)\,ds&\lesssim A_{k-1}(t_{k-1}-t)e^{\frac12A_{k-2}t_{k-1}}\\
&\lesssim A_{k-1}T\,e^{-\frac c2\frac{A_{k-2}}{A_{k-3}}}\\
&\lesssim \left(\frac{A_{k-2}}{A_0}e^{-\frac c2\frac{A_{k-2}}{A_{k-3}}}\right)\frac{A_{k-1}}{A_{k-2}}.
}
The prefactor can be made arbitrarily small with the choice of $N_0$ as in \eqref{exp_small} and we arrive at
\eqn{
\int_t^{t_{k-1}}\zeta_{k-1}(s)\,ds\leq \frac{A_{k-1}}{A_{k-2}}.
}
For the second term,
\eqn{
\int_{t_{k-1}}^0\zeta_{k-1}(s)\,ds\leq 2A_{k-1}\int_{t_{k-1}}^0e^{\frac12A_{k-2}s}\,ds\leq 4\frac{A_{k-1}}{A_{k-2}}.
}
Hence
\eqn{
\int_t^0\zeta_{k-1}(s)\,ds\leq 5\frac{A_{k-1}}{A_{k-2}}.
}
This yields \eqref{eta_global_bound}.
\end{proof}

\subsection{Inviscid blow-up}

In this section, we prove Theorem~\ref{inviscid_euler_blowup_theorem}. By translating in time, we may arrange that the blow-up occurs at $t=0$ and the initial data is specified at $t=-T$ for $T>0$ defined as in \S\ref{trapping_subsec}.
For each $K\in\mathbb N$, we consider the Galerkin truncation of \eqref{x_system_nondimensionalized} up to mode $K$; namely,
\begin{equation}\label{truncated_inviscid}
x_k'=x_{k-1}x_k
-
\delta_k x_{k+1}^2\,\mathbbm 1_{k\leq K-1},
\qquad 0\leq k\leq K,
\end{equation}
with terminal data prescribed at time $t=0$.

\begin{proposition}\label{truncated_barrier_prop}
Fix $K\in\mathbb N$ and let $\eta_k,\zeta_k,\mathcal R_K$ be as in
Definition~\ref{R_region_definition}. Then there exists a unique solution
\eqn{
x^K=(x_0^K,\dots,x_K^K)\in C^1([-T,0];\mathbb R^{K+1})
}
to \eqref{truncated_inviscid} with terminal data
\eqn{
x_k^K(0)=A_k,\qquad 0\leq k\leq K
}
where $T$ is as defined in \eqref{tk_times_def}. Moreover,
\begin{equation}\label{uniform_truncated_bounds}
0\leq x_k^K(t)\leq 2A_k
\exp\Bigl(\frac12A_{k-1}\max\{t,t_k\}\Bigr)
\quad (1\leq k\leq K),
\quad
0\leq x_0^K(t)\leq 2A_0.
\end{equation}
\end{proposition}

\begin{proof}
By standard ODE theory, there is a unique local solution backward in time from $t=0$. To continue the solution backward to $t=-T$, we consider the truncated energy
\eq{\label{truncated_e}
e_K(t)&=\sum_k\frac12N_k^{-2\alpha}(x_k^K)^2(t).
}
Then, for the lifetime of a local solution $x_k^K$,
\eqn{
e_K'(t)&=\sum_{k\leq K}N_k^{-2\alpha}x_{k-1}x_k^2-\delta_kN_k^{-2\alpha}x_kx_{k+1}^2\mathbbm1_{k\leq K-1}=0,
}
using the recurrence relation $\delta_kN_k^{-2\alpha}=N_{k+1}^{-2\alpha}$, where we have omitted the superscript $K$'s to avoid clutter. Thus $|x_k^K|\lesssim_{K}1$ so the solutions can be smoothly continued to all $t\in(-\infty,0]$.

We claim
\eq{\label{stays_trapped_in_R}
x^K(t)\in \mathcal R_K(t)
\qquad \forall\, t\in[-T,0].
}
Let us continue to omit the dependence on $K$. To argue \eqref{stays_trapped_in_R}, recall that $\mathcal R_K(t)$ is characterized by the inequality $\eta_k(t)\leq x_k(t)\leq\zeta_k(t)$ for all $k\leq K$. At $t=0$, this clearly holds as an equality. By a standard argument based on the first time of exit from the region $\mathcal R_K(t)$, it suffices to observe that
\begin{itemize}
    \item $x_k'(t)\leq\eta_k'(t)$ when $x(t)\in\mathcal R_K(t)$ and $x_k(t)=\eta_k(t)$
    \item $x_k'(t)\geq\zeta_k'(t)$ when $x(t)\in\mathcal R_K(t)$ and $x_k(t)=\zeta_k(t)$.
\end{itemize}
Note that the inequalities between the derivatives are reversed because the argument is backward in time.

First we argue that $x_k=\zeta_k$ implies $x_k'\geq\zeta_k'$. This is clear from the definition for $k=0$. Consider the case where $k\geq1$ and $t\in[t_k,0]$. Then, from $x(t)\in\mathcal R_K(t)$, we have
\eqn{
x_k'&=x_{k-1}x_k-\delta_kx_{k+1}^2\mathbbm1_{k\leq K-1}\geq \eta_{k-1}\zeta_k-\delta_k\zeta_{k+1}^2\mathbbm1_{k\leq K-1}
}
and the claim follows from \eqref{eta_quarter_bound}.

Consider the remaining case, in which $k\geq1$ and $t\in[-T,t_k)$. By definition then, $\zeta_k'=0$. We may assume $k\geq3$ because otherwise, $[-T,t_k)$ is empty. Further, the $k=K$ case is trivial because the negative term in \eqref{truncated_inviscid} vanishes. The requirement is then
\eqn{
x_k'&=x_{k-1}x_k-\delta_k x_{k+1}^2\geq0
}
for $k\in\{1,2,\ldots,K-1\}$. Since $x$ has positive components, we estimate the ratio. By Lemma~\ref{barrier_bounds_lemma}, and using that $x\in\mathcal R_K$, we have
\eqn{
\frac{x_{k-1}x_k}{\delta_kx_{k+1}^2}\geq\frac{\eta_{k-1}\eta_k}{\delta_k\zeta_{k+1}^2}\gtrsim \frac{A_{k-1}A_k}{\delta_kA_{k+1}^2}\exp(-5\frac{A_{k-1}}{A_{k-2}}-5\frac{A_{k-2}}{A_{k-3}}-A_kt_{k+1}).
}
By \eqref{ratios}, the exponent obeys
\eqn{
-5\frac{A_{k-1}}{A_{k-2}}-5\frac{A_{k-2}}{A_{k-3}}-A_kt_{k+1}&=-5\frac{A_{k-1}}{A_{k-2}}-5\frac{A_{k-2}}{A_{k-3}}+c\frac{A_k}{A_{k-1}}\geq\frac c2\frac{A_k}{A_{k-1}}.
}
Choosing $N_0$ large as necessary so that the exponential defeats the polynomial prefactor as in \eqref{exp_small}, we conclude that $x_{k-1}x_k\geq\delta_kx_{k+1}^2$.

It is immediate from the definition that $x_k=\eta_k$ implies $x_k'\leq\eta_k'$. This completes the proof of \eqref{stays_trapped_in_R}, from which the claim \eqref{uniform_truncated_bounds} is immediate.
\end{proof}

\begin{proof}[Proof of Theorem~\ref{inviscid_euler_blowup_theorem}]
Fix $s>0$ and choose $\beta$ as in \eqref{inviscid_A_assumptions}. Let $A_k=N_k^\beta$ and let
$T=c/A_0$ be as above. For each $K$, let $x^K$ be the solution given by
Proposition~\ref{truncated_barrier_prop}. From \eqref{truncated_inviscid} and the uniform-in-$K$ bound \eqref{uniform_truncated_bounds}, we have $\|x_k^K\|_{C^1([-T,0])}\lesssim_k1$ from which one further bootstraps $\|x_k^K\|_{C^2([-T,0])}\lesssim_k1$. By diagonalization and Arzel\`a--Ascoli, one can pass to a subsequence that converges in $C^1$ for each $k$. Because the system \eqref{x_system_nondimensionalized} is local in $k$, i.e., $x_k'$ depends only on $x_{k-1},x_k,x_{k+1}$, this suffices to conclude that the limit $x_k$ obeys \eqref{x_system_nondimensionalized} with
\eqn{
x_k(0)=A_k
\qquad \forall\, k\geq 0.
}
Thus, $\|X(0)\|_{\mathcal{C}^s}=\sup_kN_k^{\beta-\alpha+s}=\infty$ due to the assumption $\beta>-s+\alpha$. Moreover, from \eqref{uniform_truncated_bounds}, we have
\eq{\label{final_bounds}
0\leq X_k(t)\leq2\frac{A_k}{N_k^\alpha}\exp\Big(\frac12A_{k-1}\max\{t,t_k\}\Big)\qquad \forall k\geq1.
}

What remains is to show that $X(t)\in \mathcal{C}^\infty$ for $t\in[-T,0)$, from which we can infer that a finite-time blow-up indeed occurs at $t=0$. Toward this, fix any $\sigma>0$ and $t\in[-T,0)$, and let $k_*\in\mathbb N$ such that $t_{k_*}>t$. To verify $X(t)\in \mathcal{C}^\sigma$, it suffices to check the high frequencies. By \eqref{final_bounds} and \eqref{exp_small}, we have
\eqn{
\sup_{k>k_*}N_k^\sigma|X_k|\lesssim\sup_{k>k_*}A_kN_k^{\sigma-\alpha}\exp\Big(-\frac c2\frac{A_{k-1}}{A_{k-2}}\Big)<\infty.
}
\end{proof}

\subsection{Viscous blow-up}

As discussed in \S\ref{viscous_discussion}, we proceed in the viscous case by introducing a mild force to prevent the dissipation from acting on each mode before its activation time. To this end, fix a decreasing cutoff function $\rho\in C_c^\infty([0,1))$ with $\rho\equiv1$ on $[0,\frac12]$, and define $\rho_k(t)\coloneqq \rho(t/t_k)$ for $k\geq1$, and $\rho_0\equiv1$. Thus, 
\eqn{
\rho_k\equiv1\text{ for }t\in[t_k/2,0],\quad \supp\rho_k\subset[t_k,0]\qquad\forall k\geq1.
}
Consider the system
\begin{equation}\label{truncated_viscous}
x_k'=-\rho_k(t)N_k^2x_k+x_{k-1}x_k
-
\delta_k x_{k+1}^2\,\mathbbm 1_{k\leq K-1},
\qquad 0\leq k\leq K
\end{equation}
which is a truncation of the forced viscous system \eqref{x_system_nondimensionalized} with $\nu=1$, having set $g_k=(1-\rho_k(t))N_k^2x_k$. The strategy toward Theorem~\ref{viscous_ns_blowup_theorem} is to construct solutions of \eqref{truncated_viscous} using nearly the same trapping region as in the inviscid case, then to estimate the resulting force. We claim that a statement identical to Proposition~\ref{truncated_barrier_prop} indeed holds for solutions of \eqref{truncated_viscous}.

\begin{proposition}\label{truncated_barrier_prop_2}
Fix $K\in\mathbb N$ and let $\eta_k,\zeta_k,\mathcal R_K$ be as in
Definition~\ref{R_region_definition}. Then there exists a unique solution
\eqn{
x^K=(x_0^K,\dots,x_K^K)\in C^1([-T,0];\mathbb R^{K+1})
}
to \eqref{truncated_viscous} with terminal data
\eqn{
x_k^K(0)=A_k,\qquad 0\leq k\leq K.
}
Moreover,
\begin{equation}\label{uniform_truncated_bounds_viscous}
0\leq x_k^K(t)\leq 2A_k
\exp\Bigl(\frac12A_{k-1}\max\{t,t_k\}\Bigr)
\quad (1\leq k\leq K),
\qquad
0\leq x_0^K(t)\leq 2A_0.
\end{equation}
\end{proposition}

\begin{proof}
We once again invoke standard ODE theory to obtain a unique local solution backward in time from $t=0$. Define the truncated energy as in \eqref{truncated_e}. Now the energy identity for $x_k^K$ takes the form
\eqn{
e_K'(t)+\sum_{k\leq K}\rho_k(t)N_k^{2-2\alpha}x_k^2&=\sum_{k\leq K}N_k^{-2\alpha}x_{k-1}x_k^2-\delta_kN_k^{-2\alpha}x_kx_{k+1}^2\mathbbm1_{k\leq K-1}=0,
}
once again leaving the dependence on $K$ implicit. It follows that
\eqn{
e_K'(t)&=-\sum_{k\leq K}\rho_k(t)N_k^{2-2\alpha}x_k^2\geq -2N_K^2e_K(t).
}
Thus
\eqn{
e_K(t)\leq e^{-2N_K^2t}e_K(0)\leq e^{2N_K^2T}e_\infty(0)\lesssim_K1
}
for all $t\leq0$ in the lifetime of the local solution, from which it follows that $x_k(t)$ stays bounded on $[-T,0]$, for each $k$. Thus the solutions can be smoothly continued to the whole time interval.

We argue \eqref{uniform_truncated_bounds_viscous} by way of \eqref{stays_trapped_in_R} analogously to the inviscid case. The argument is identical whenever $\rho_k$ vanishes, so we are left with only two cases:
\begin{itemize}
    \item $k=0$ and $t\in[-T,0]$, or
    \item $k\geq1$ and $t\in[t_k,0]$.
\end{itemize}
In the first case, note that inclusion of the viscosity only affects the upper barrier backward in time, and that this effect has already been incorporated into the definition of $\zeta_0(t)$ in Definition~\ref{R_region_definition}; thus the argument is identical. In the second case, from $x(t)\in\mathcal R_K(t)$, we have
\eqn{
	\eta_{k-1}-\rho_k(t)N_k^2\geq \frac34A_{k-1}-N_k^2\geq\frac12A_{k-1},
}
using \eqref{eta_quarter_bound} and that $\beta>2b$, and taking $N_0$ large as needed. Thus,
\eqn{
x_k'&=-\rho_k(t)N_k^2x_k+x_{k-1}x_k-\delta_kx_{k+1}^2\mathbbm1_{k\leq K-1}\geq (\eta_{k-1}-\rho_k(t)N_k^2)\zeta_k-\delta_k\zeta_{k+1}^2\mathbbm1_{k\leq K-1}
}
and the claim follows.
\end{proof}

\begin{proof}[Proof of Theorem~\ref{viscous_ns_blowup_theorem}]
Existence of a solution $x(t)$ of
\eqn{
x_k'=-\rho_k(t)N_k^2x_k+x_{k-1}x_k-\delta_kx_{k+1}^2
}
that blows up at $t=0$ and obeys
\begin{equation}\label{sol_bounds}
0\leq x_k(t)\leq 2A_k
\exp\Bigl(\frac12A_{k-1}\max\{t,t_k\}\Bigr)
\quad (k\geq1),
\qquad
0\leq x_0(t)\leq 2A_0
\end{equation}
is proved exactly as in the inviscid case, using the uniform estimate~\eqref{uniform_truncated_bounds_viscous}. What remains is to estimate the force. Clearly, $x(t)$ satisfies the desired system \eqref{x_system_nondimensionalized} with force
\eqn{
	g_0(t)=0,\qquad g_k(t)=(1-\rho_k(t))N_k^2x_k\quad\text{for }k\geq1.
}
It suffices to show the desired inclusion for $g$, since the force for the original system \eqref{l2_obukhov} comes from the rescaling $f_k=N_k^{-\alpha}g_k$. Observe that $1-\rho_k$ is supported in $(-\infty,t_k/2]$. Thus, for $k\geq3$, \eqref{sol_bounds} gives
\eqn{
	|g_k(t)|&\lesssim \sup_{t\leq t_k/2}N_k^{2}A_k\exp\Big(\frac12A_{k-1}t\Big)=N_k^{2}A_k\exp\Big(-\frac c4\frac{A_{k-1}}{A_{k-2}}\Big),
}
which clearly lies in $\mathcal C^\infty$ by the rapid growth of $A_k$; see \eqref{exp_small}. For the higher derivatives, we apply the Leibniz rule to $g_k(t) = (1-\rho_k(t))N_k^2 x_k(t)$. Recall that $|\d_t^j \rho_k(t)| \lesssim_j |t_k|^{-j} \sim A_{k-2}^j$. Differentiating the system shows that time derivatives of $x_k$ grow at most polynomially in the local amplitudes $(A_{k+i})_{i\leq j}$ and frequencies $(N_{k+i})_{i\leq j}$. As in the $j=0$ case, the $1-\rho_k$ factor restricts the support of $g_k$ to $t \leq t_k/2$, meaning all bounded terms are weighted by the small factor $\exp(-\frac{c}{4}\frac{A_{k-1}}{A_{k-2}})$, which once again defeats the polynomial factors. The extreme decay of the exponential ensures that $g$, and therefore $f_k=N_k^{-\alpha}g_k$, lies in $C_t^\infty([-T,0];\mathcal{C}^\infty)$. Since each $f_k$ vanishes in a neighborhood of $t=0$, we may extend it by $0$ to obtain an $f\in C_t^\infty([-T,\infty);\mathcal{C}^\infty)$ as claimed in Theorem~\ref{viscous_ns_blowup_theorem}.
\end{proof}

\section{Prospects for transfer to the PDE setting}\label{prospects_section}

We conclude with an argument that the Obukhov model with super-exponential frequency growth is a plausible candidate to be embedded in the full 3D Navier--Stokes equations to complete the two-step program mentioned in \S\ref{dyadic_models_sec}. By an ``embedding,'' we mean either of two types of results. First, there is a precise notion of embedding, as appears in Tao's work (\cite{MR3486169}, \cite{MR3702540}, \cite{MR3809006}, etc.), in which there is some map that precisely takes solutions of the ODE system to solutions of the PDE. There is also a more relaxed notion that appeared, for instance, in~\cite{MR5008166} where only particular solutions of a dyadic model are realized by some appropriate projection of the PDE solutions, up to various controlled errors.

As detailed in \S\ref{dyadic_models_sec}, the only other model of the 3D Navier--Stokes equations that is known to blow up is the one invented by Tao in~\cite{MR3486169}. We point out two meaningful advantages of the model \eqref{l2_obukhov} as a candidate for transfer to the Navier--Stokes. The first is that the frequency separation is allowed to be extremely wide, which facilitates minimizing the interactions between the solution components at various scales, and has been instrumental in nearly all multi-scale constructions in fluid mechanics. In comparison, the frequencies in~\cite{MR3486169} cannot be significantly more separated than $N_k\sim 2^k$ due to the use of the KP-type low-low-high interaction $P_k\div u_{k-1}\otimes u_{k-1}$. Without a large parameter to separate the scales, it would seem to be extremely difficult to control errors in an embedding construction.

The other advantage of \eqref{l2_obukhov} is that it contains only nonlinear interactions that are organically represented in the true Euler/Navier--Stokes nonlinearity, as discussed in \S\ref{motivation_section}. The high-high-low interaction $-N_k^\alpha X_{k+1}^2$, corresponding to $P_k\div u_{k+1}\otimes u_{k+1}$, can act in any prescribed way, as was discovered by De Lellis and Sz\'ekelyhidi in \cite{MR3090182} and taken advantage of in~\cite{MR5008166,palasek2025arbitrary,cheskidov2025instantaneous}, among numerous convex integration constructions. The other Obukhov interaction, $N_{k-1}^\alpha X_{k-1}X_k$, is more difficult to harness, but nonetheless is easily understood as the nonlinear interaction $u_k\cdot\grad u_{k-1}$. For comparison, in~\cite{MR3486169}, there is a complicated array of nonlinear interactions with different weights that do not have any clear counterpart in the PDE nonlinearity; see Figure~\ref{fig:models}.

We remark finally that, in the inviscid setting, shell models exhibiting turbulent cascades (e.g., the KP model) suggest a different possible route toward Euler blow-up. That route is distinct from the mechanism developed here, which relies instead on strong scale separation and suppression of the high-high-low interaction.

\bibliographystyle{abbrv}
\bibliography{refs}

\end{document}

%% file: fig.tex
\begin{tikzpicture}[
    box/.style={draw, thick, rectangle, minimum width=2.5em, minimum height=1.6em, align=center},
    filled arrow/.style={thick, -{Triangle[angle=45:8pt]}},
    open arrow/.style={thick, -{Triangle[open, angle=45:8pt]}},
    line/.style={thick},
    loop/.style={
        thick,
        decoration={
            markings,
            mark=at position 0.15 with {\arrow{Triangle[angle=45:5pt]}},
            mark=at position 0.65 with {\arrow{Triangle[angle=45:5pt]}}
        },
        postaction={decorate}
    },
    lbl/.style={font=\scriptsize, text=black, inner sep=3pt} 
]

\draw[->, thick] (-1.5, 1.6) -- (-1.5, 5.4) node[pos=0.5, above, sloped] {\small frequency ($\log$ scale)};

\def\xKP{0}
\def\xObE{2.9}
\def\xObS{5.8}
\def\xTao{8.4}

\fill[gray!25] (\xObS - 1.4, -2.6) rectangle (\xObS + 1.4, -1.4);
\node[below, font=\small, text=black, align=center] at (\xObS, -2.7) {\textbf{Theorems~\ref{viscous_ns_blowup_theorem}, \ref{inviscid_euler_blowup_theorem}}};


\node[box] (K0) at (\xKP, 0) {$X_0$};
\node[box] (K1) at (\xKP, 1.4) {$X_1$};
\node[box] (K2) at (\xKP, 2.8) {$X_2$};
\node[box] (K3) at (\xKP, 4.2) {$X_3$};
\node[box] (K4) at (\xKP, 5.6) {$X_4$};
\node[box] (K5) at (\xKP, 7.0) {$X_5$};

\draw[filled arrow] (K0) -- node[lbl, left] {$N_0^\alpha$} (K1);
\draw[filled arrow] (K1) -- node[lbl, left] {$N_1^\alpha$} (K2);
\draw[filled arrow] (K2) -- node[lbl, left] {$N_2^\alpha$} (K3);
\draw[filled arrow] (K3) -- node[lbl, left] {$N_3^\alpha$} (K4);
\draw[filled arrow] (K4) -- node[lbl, left] {$N_4^\alpha$} (K5);

\node at (\xKP, 8.0) {$\vdots$};

\node[box] (OE0) at (\xObE, 0) {$X_0$};
\node[box] (OE1) at (\xObE, 1.4) {$X_1$};
\node[box] (OE2) at (\xObE, 2.8) {$X_2$};
\node[box] (OE3) at (\xObE, 4.2) {$X_3$};
\node[box] (OE4) at (\xObE, 5.6) {$X_4$};
\node[box] (OE5) at (\xObE, 7.0) {$X_5$};

\draw[open arrow] (OE0) -- node[lbl, left] {$N_0^\alpha$} (OE1);
\draw[open arrow] (OE1) -- node[lbl, left] {$N_1^\alpha$} (OE2);
\draw[open arrow] (OE2) -- node[lbl, left] {$N_2^\alpha$} (OE3);
\draw[open arrow] (OE3) -- node[lbl, left] {$N_3^\alpha$} (OE4);
\draw[open arrow] (OE4) -- node[lbl, left] {$N_4^\alpha$} (OE5);

\node at (\xObE, 8.0) {$\vdots$};

\node[box] (OS0) at (\xObS, 0) {$X_0$};
\node[box] (OS1) at (\xObS, 1.2) {$X_1$};
\node[box] (OS2) at (\xObS, 3.4) {$X_2$};
\node[box] (OS3) at (\xObS, 7.) {$X_3$};

\draw[open arrow] (OS0) -- node[lbl, left] {$N_0^\alpha$} (OS1);
\draw[open arrow] (OS1) -- node[lbl, left] {$N_1^\alpha$} (OS2);
\draw[open arrow] (OS2) -- node[lbl, left] {$N_2^\alpha$} (OS3);

\node at (\xObS, 8.) {$\vdots$};

\def\xTaoR{\xTao + 2.5} 

\node[box] (T10) at (\xTao, 0) {$X_{1,0}$};
\node[box] (T20) at (\xTaoR, 0) {$X_{2,0}$};
\node[box] (T30) at (\xTaoR, 1.0) {$X_{3,0}$};
\node[box] (T40) at (\xTao, 2.0) {$X_{4,0}$};

\draw[loop] (\xTao, 1.0) ellipse (0.15 and 0.65);
\draw[line] (\xTao + 0.15, 1.0) -- node[lbl, above, pos=0.55] {$N_0^\alpha\epsilon^{-2}$} (T30.west);

\draw[filled arrow] (T10) -- node[lbl, below] {$N_0^\alpha\epsilon$} (T20);
\draw[open arrow] (T20) -- node[lbl, right] {$N_0^\alpha\epsilon^{-1}K^{10}$} (T30);
\draw[filled arrow] (T10) -- node[lbl, above left, pos=0.2, xshift=30pt] {$N_0^\alpha\epsilon^2e^{-K^{10}}$} (T30);

\node[box] (T11) at (\xTao, 3.6) {$X_{1,1}$};
\node[box] (T21) at (\xTaoR, 3.6) {$X_{2,1}$};
\node[box] (T31) at (\xTaoR, 4.6) {$X_{3,1}$};
\node[box] (T41) at (\xTao, 5.6) {$X_{4,1}$};

\draw[filled arrow] (T40) -- node[lbl, left] {$N_1^\alpha K$} (T11);

\draw[loop] (\xTao, 4.6) ellipse (0.15 and 0.65);
\draw[line] (\xTao + 0.15, 4.6) -- node[lbl, above, pos=0.55] {$N_1^\alpha\epsilon^{-2}$} (T31.west);

\draw[filled arrow] (T11) -- node[lbl, below] {$N_1^\alpha\epsilon$} (T21);
\draw[open arrow] (T21) -- node[lbl, right] {$N_1^\alpha\epsilon^{-1}K^{10}$} (T31);
\draw[filled arrow] (T11) -- node[lbl, above left, pos=0.2, xshift=30pt] {$N_1^\alpha\epsilon^2e^{-K^{10}}$} (T31);

\node[box] (T12) at (\xTao, 7.2) {$X_{1,2}$};
\draw[filled arrow] (T41) -- node[lbl, left] {$N_2^\alpha K$} (T12);
\node at (\xTao, 8.2) {$\vdots$};

\def\yTitle{-1.0}
\def\yInv{-1.7}
\def\yVisc{-2.3}
\def\xLeft{-3.1}

\begin{scope}[every node/.style={font=\small}]

\node[anchor=west] at (\xLeft + 0.1, \yInv) {\textbf{Inviscid}};
\node[anchor=west] at (\xLeft + 0.1, \yVisc) {\textbf{Viscous}};

\node[align=center] at (\xKP, \yTitle) {\textbf{KP model}};
\node[align=center] at (\xObE, \yTitle) {\textbf{Ob.} ($N_k=\lambda^k$)};
\node[align=center] at (\xObS, \yTitle) {\textbf{Ob.} ($N_k\gg\lambda^k$)};
\node[align=center] at ({(\xTao + \xTaoR)/2}, \yTitle) {\textbf{Tao~\cite{MR3486169} model}};

\node at (\xKP, \yInv) {blow-up};
\node at (\xObE, \yInv) {regularity*};
\node at (\xObS, \yInv) {blow-up};
\node at ({(\xTao + \xTaoR)/2}, \yInv) {blow-up};

\node at (\xKP, \yVisc) {regularity*};
\node at (\xObE, \yVisc) {regularity*};
\node at (\xObS, \yVisc) {blow-up};
\node at ({(\xTao + \xTaoR)/2}, \yVisc) {blow-up};

\end{scope}

\draw[thick] (\xLeft, -1.4) -- (11.3, -1.4);
\draw (\xLeft, -2.0) -- (11.3, -2.0);
\draw[thick] (\xLeft, -2.6) -- (11.3, -2.6);
\end{tikzpicture}

%% file: main.bbl
\begin{thebibliography}{10}

\bibitem{MR2844828}
D.~Barbato, F.~Morandin, and M.~Romito.
\newblock Smooth solutions for the dyadic model.
\newblock {\em Nonlinearity}, 24(11):3083--3097, 2011.

\bibitem{MR2415066}
A.~Cheskidov.
\newblock Blow-up in finite time for the dyadic model of the {N}avier-{S}tokes equations.
\newblock {\em Trans. Amer. Math. Soc.}, 360(10):5101--5120, 2008.

\bibitem{MR4607726}
A.~Cheskidov, M.~Dai, and S.~Friedlander.
\newblock Dyadic models for fluid equations: a survey.
\newblock {\em J. Math. Fluid Mech.}, 25(3):Paper No. 62, 26, 2023.

\bibitem{cheskidov2025instantaneous}
A.~Cheskidov, M.~Dai, and S.~Palasek.
\newblock Instantaneous type {I} blow-up and non-uniqueness of smooth solutions of the {N}avier--{S}tokes equations.
\newblock {\em arXiv preprint arXiv:2511.09556}, 2025.

\bibitem{MR2337019}
A.~Cheskidov, S.~Friedlander, and N.~Pavlovi\'c.
\newblock Inviscid dyadic model of turbulence: the fixed point and {O}nsager's conjecture.
\newblock {\em J. Math. Phys.}, 48(6):065503, 16, 2007.

\bibitem{MR2600714}
A.~Cheskidov, S.~Friedlander, and N.~Pavlovi\'c.
\newblock An inviscid dyadic model of turbulence: the global attractor.
\newblock {\em Discrete Contin. Dyn. Syst.}, 26(3):781--794, 2010.

\bibitem{MR5008166}
M.~P. Coiculescu and S.~Palasek.
\newblock Non-uniqueness of smooth solutions of the {N}avier--{S}tokes equations from critical data.
\newblock {\em Invent. Math.}, 244(1):165--219, 2026.

\bibitem{MR4929621}
D.~C\'ordoba, L.~Martinez-Zoroa, and F.~Zheng.
\newblock Finite time singularities to the 3{D} incompressible {E}uler equations for solutions in {$C^\infty (\Bbb R^3 \setminus \{0\})\cap C^{1,\alpha}\cap L^2$}.
\newblock {\em Ann. PDE}, 11(2):Paper No. 19, 56, 2025.

\bibitem{MR3090182}
C.~De~Lellis and L.~Sz\'ekelyhidi, Jr.
\newblock Dissipative continuous {E}uler flows.
\newblock {\em Invent. Math.}, 193(2):377--407, 2013.

\bibitem{MR2038114}
S.~Friedlander and N.~Pavlovi\'c.
\newblock Blowup in a three-dimensional vector model for the {E}uler equations.
\newblock {\em Comm. Pure Appl. Math.}, 57(6):705--725, 2004.

\bibitem{MR3339169}
I.-J. Jeong and D.~Li.
\newblock A blow-up result for dyadic models of the {E}uler equations.
\newblock {\em Comm. Math. Phys.}, 337(2):1027--1034, 2015.

\bibitem{MR2095627}
N.~H. Katz and N.~Pavlovi\'c.
\newblock Finite time blow-up for a dyadic model of the {E}uler equations.
\newblock {\em Trans. Amer. Math. Soc.}, 357(2):695--708, 2005.

\bibitem{MR2180809}
A.~Kiselev and A.~Zlato\v{s}.
\newblock On discrete models of the {E}uler equation.
\newblock {\em Int. Math. Res. Not.}, (38):2315--2339, 2005.

\bibitem{MR3469428}
P.~G. Lemari\'e-Rieusset.
\newblock {\em The {N}avier-{S}tokes problem in the 21st century}.
\newblock CRC Press, Boca Raton, FL, 2016.

\bibitem{leray}
J.~Leray.
\newblock Sur le mouvement d'un liquide visqueux emplissant l'espace.
\newblock {\em Acta Math.}, 63(1):193--248, 1934.

\bibitem{looi}
S.~Looi.
\newblock {\em to appear}, 2026.

\bibitem{obukhov}
A.~Obukhov.
\newblock Some general properties of equations describing the dynamics of the atmosphere.
\newblock {\em Academy of Sciences, USSR, Izvestiya, Atmospheric and Oceanic Physics}, 7:471--475, 1971.

\bibitem{palasek2025arbitrary}
S.~Palasek.
\newblock Arbitrary norm growth in the 3{D} {N}avier--{S}tokes equations.
\newblock {\em arXiv preprint arXiv:2509.18595}, 2025.

\bibitem{MR4992011}
S.~Palasek.
\newblock Non-uniqueness in the {L}eray-{H}opf class for a dyadic {N}avier-{S}tokes model.
\newblock {\em Int. Math. Res. Not. IMRN}, (22):Paper No. rnaf344, 32, 2025.

\bibitem{MR3486169}
T.~Tao.
\newblock Finite time blowup for an averaged three-dimensional {N}avier-{S}tokes equation.
\newblock {\em J. Amer. Math. Soc.}, 29(3):601--674, 2016.

\bibitem{MR3702540}
T.~Tao.
\newblock On the universality of potential well dynamics.
\newblock {\em Dyn. Partial Differ. Equ.}, 14(3):219--238, 2017.

\bibitem{MR3809006}
T.~Tao.
\newblock On the universality of the incompressible {E}uler equation on compact manifolds.
\newblock {\em Discrete Contin. Dyn. Syst.}, 38(3):1553--1565, 2018.

\bibitem{MR4196152}
T.~Tao.
\newblock On the universality of the incompressible {E}uler equation on compact manifolds, {II}. {N}on-rigidity of {E}uler flows.
\newblock {\em Pure Appl. Funct. Anal.}, 5(6):1425--1443, 2020.

\bibitem{MR2231615}
F.~Waleffe.
\newblock On some dyadic models of the {E}uler equations.
\newblock {\em Proc. Amer. Math. Soc.}, 134(10):2913--2922, 2006.

\end{thebibliography}
